\documentclass[12pt]{article}
\usepackage{amsmath,amssymb, amsfonts, amsthm,amscd}
\usepackage[T2A]{fontenc}
\usepackage[cp1251]{inputenc}
\usepackage[english]{babel}
\usepackage{graphicx}

\theoremstyle{plain}
\newtheorem{theorem}{Theorem}

\newtheorem{proposition}{Proposition}

\theoremstyle{definition}
\newtheorem{definition}{Definition}

\begin{document}

\begin{center}
{\huge Almost zip Bezout domain}
\end{center}
\vskip 0.1cm \centerline{{\Large Bohdan Zabavsky, \; Oleh Romaniv}}

\vskip 0.3cm

\centerline{\footnotesize{Department of Mechanics and Mathematics, Ivan Franko National University}}
\centerline{\footnotesize{Lviv, 79000, Ukraine}}
 \centerline{\footnotesize{zabavskii@gmail.com, \; oleh.romaniv@lnu.edu.ua}}
\vskip 0.5cm

\centerline{\footnotesize{March, 2019}}
\vskip 0.7cm

\footnotesize{\noindent\textbf{Abstract:} \textit{J. Zelmanowitz introduced the concept of ring, which we call zip rings. In this paper we characterize a commutative Bezout domain whose finite homomorphic images are zip rings modulo its nilradical.} }

\vspace{1ex}
\footnotesize{\noindent\textbf{Key words and phrases:} \textit{Bezout ring; elementary divisor ring; zip ring; $J$-Noetherian domain.}

}

\vspace{1ex}

\vspace{1,5truecm}

\normalsize

\section{Introduction}

All rings considered will be commutative with identity. A ring is a Bezout ring if every finitely generated ideal is principal. I.~Kaplansky \cite{2KaplEDMod} defined the class of elementary divisor rings as rings $R$ for which every matrix $A$ over $R$ admits a diagonal reduction, that is there exist invertible matrices $P$ and $Q$ such that $PAD$ is a diagonal matrix $D=(d_i)$ with the property that every $d_i$ is a divisor $d_{i+1}$. B.~Zabavsky defined  fractionally regular rings as rings $R$ such that for which every nonzero and nonunit element a from $R$ the classical quotient ring $Q_{cl}(R/\mathrm{rad}(aR))$ is regular, where $\mathrm{rad}(aR)$ is nilradical of $aR$ \cite{3Zabavsky}. We say that the ring $R$ has stable range 2 if whenever $aR+bR+cR=R$, then there are $\lambda,\mu\in R$ such that $(a+c\lambda)R+(b+c\mu)R=R$. We say $R$ is semi-prime if $\mathrm{rad}(R)=\{0\}$, where $\mathrm{rad}(R)$ is the nilradical of the ring $R$. Obviously, rings in which nonzero principal ideal has only finitely many minimal prime are examples of fractionally regular rings \cite{5Anderson}.

 An ideal $I$ of a ring $R$ is called a $J$-radical if it is an intersection of maximal ideals, or, equivalently, if $R/I$
has zero Jacobson radical.  We call $R$ $J$-Noetherian if it satisfies the ascending chain condition on $J$-radical ideals.

For every ideal $I$ i $R$ we define the annihilator of $I$ by $I^\bot=\{x\in R\mid ix=0 \;  \forall i\in I\}$.

Following C.~Faith \cite{4Faith} a ring $R$ is zip if $I$ is an ideal and if $I^\bot=\{0\}$ than $I_0^\bot=\{0\}$ for a finitely generated ideal $I_0\subset I$. An ideal $I$ of a ring $R$ is dense if its annihilator is zero. Thus $I$ is a dense ideal if and only if it is a faithful $R$-module. A ring $R$ is a Kasch ring if $I^\bot\ne \{0\}$ for any ideal $I\ne R$.

Let $R$ be a ring. Then the ring $R$ has finite Goldie dimension if it contains a direct sum of  finite number of nonzero ideals. A ring $R$ is called a Goldie ring if it has finite Goldie dimension and satisfies the ascending chain condition for annihilators \cite{4Faith,6Kaplansky,1Zelmanowitz}. By~\cite{4Faith} we have the following  result.

\begin{theorem} \cite{4Faith} \label{theor-1}
    Semiprime commutative ring $R$ is zip if and only if $R$ is a Goldie ring.
\end{theorem}

\begin{proposition} \cite{4Faith}
    A commutative Kasch ring is zip.
\end{proposition}

\begin{proposition} \cite{4Faith}
    If $Q_{cl}(R)$ is a Kasch ring then $R$ is zip.
\end{proposition}

For further research we will need the following results.

\begin{theorem} \cite{4Faith}
    A commutative ring $R$ is zip if and only if its classical ring of quotients $Q_{cl}(R)$  is zip.
\end{theorem}

\begin{theorem} \label{theor-1.3new}
    Let $R$ be a commutative Bezout domain and $0\ne a\in R$, then $R/aR$ is a Kasch ring if and only if $R$ is a ring in which any maximal ideal is principal.
\end{theorem}

\begin{proof}
    First we will prove that the annihilator of any principal ideal of $R/aR$ is a principal ideal.

    Suppose $b\in R$ and $aR\subseteq bR$. Then $(b\colon a)=\{r\in R\mid br\in aR\}=sR$, where $a=bs$, so $(b\colon a)=aR$. We can also show that every principal ideal of $R/aR$ is a annihilator of a principal ideal. Moreover, if $I_1=\mathrm{Ann}(J_1)$, $I_2=\mathrm{Ann}(J_2)$, where $I_i$, $J_i$, $i=1,2$, are principal ideals, then
    \begin{multline*}
    \mathrm{Ann}(I_1\cap I_2)=\mathrm{Ann}(\mathrm{Ann}(J_1)\cap \mathrm{Ann}(J_2))=\\=\mathrm{Ann}(\mathrm{Ann}(J_1+J_2))=J_1+J_2=\mathrm{Ann}(J_1)+\mathrm{Ann}(J_2).
    \end{multline*}

    Let $R/aR$ be a Kasch ring. Let $\overline{M}$ be a maximal ideal in $R/aR$. Denote $R/aR=\overline{R}$. Then $\mathrm{Ann}(\overline{M})=\overline{H}$, where $\overline{H}$ is an ideal in $\overline{R}=R/aR$ and $\overline{H}=\{\overline{0}\}$. Since $\overline{H}$ annihilates the maximal ideal $\overline{M}$ then $\overline{H}\cdot \overline{M}=\{\overline{0}\}$. Since the maximal ideal $\overline{M}$ belongs to $\mathrm{Ann}(\overline{H})$, than by maximality of $\overline{M}$, $\overline{M}=\mathrm{Ann}(\overline{H})\ne R/aR$.

    Since $\overline{M}$ is a maximal ideal, then for every element $\overline{d}\ne\overline{0}$, which belongs to $\overline{H}$. We have the equality $\overline{d}\overline{M}=\{\overline{0}\}$. Thus, the maximal ideal $\overline{M}$ belongs to $\mathrm{Ann}(\overline{d})$, where $\overline{d}$ is a nonunit.

    Hence $\overline{M}=\mathrm{Ann}(\overline{d})=\overline{b}\overline{R}$. Therefore, $\overline{M}=\overline{b}\overline{R}$ and $M=bR+aR=cR$, because $R$ is a commutative Bezout domain for some $c\in R$. Hence $M$ is a maximal ideal which is a principal ideal.

    Suppose that a maximal $M$ contains an element $a$, is a principal one considering its homomorphic image we have $\overline{M}=\overline{m}\overline{R}=\mathrm{Ann}(\overline{n}\overline{R})$. Since $\overline{m}\notin U(\overline{R})$ then we have $\mathrm{Ann} (\overline{n}\overline{R})\ne \overline{R}$ and hence $\overline{n}\overline{R}\ne\{\overline{0}\}$.

    As a result $\mathrm{Ann}(\overline{M})=\mathrm{Ann}(\mathrm{Ann}(\overline{n}\overline{R})) =\overline{n}\overline{R}\ne(\overline{0})$. Therefore, $\mathrm{Ann}(\overline{M})$ is a nonzero principal ideal. This proves the fact that $\overline{R}$ is a Kasch ring.
\end{proof}

\section{Our results}

Note that

\begin{proposition}
    Let $R$ be a Bezout  ring. Then $R$ is zip of and only if every dense ideal contains a regular element.
\end{proposition}

\begin{proof}
   If $I$ is a dense ideal of a zip ring, and if $I$ is principal dense ideal contained in $I$, hence $I$ is generated by a regular element.
\end{proof}

\begin{theorem} \label{theor-6}
    Let $R$ be a semiprime commutative Bezout ring which is a Goldie ring. Then any minimal prime ideal of $R$ is principal, generated by  an idempotent, and there are only finitely many minimal prime ideals.
\end{theorem}

\begin{proof}
   The restrictions on $R$ imply that the classical quotient ring $Q_{cl}(R)$ is an Artinian regular ring with finitely many minimal prime ideals. Let $P$ be a minimal prime ideal of $R$. Consider the ideal
   $P_Q=\{\frac{p}{s}\mid p\in P\}$. It is obvious that $P_Q$ is a prime ideal of $Q_{cl}(R)$. Since $Q_{cl}(R)$ is an Artinian regular ring, there exists an idempotent $e\in Q_{cl}(R)$ such that $P_Q=eQ_{cl}(R)$. Since $R$ is arithemical ring, then we have $e\in R$ \cite{3Zabavsky}. For any $p\in P$ we obtain that $p=er$, where $r$ is a von Neumann regular element, i.e. $rxr=r$ for some $x\in R$. Hence $ep=e^2r=er=p$, we have $P\subset eR$, $e\in P$, so $eR\subset P$ and $P=eR$. Since any minimal prime ideal of $R$ is principal by \cite{5Anderson}, we have that $R$ have finitely many minimal prime ideals.
\end{proof}

\begin{definition}
    Let $R$ be  a commutative Bezout domain. Nonzero and nonunit element $a\in R$ is said to be almost zip element if $R/\mathrm{rad}(aR)$  is a zip ring. Commutative Bezout domain is said to be almost zip ring if any nonzero nonunit element of $R$ is almost zip element.
\end{definition}

\begin{theorem}\label{theor-7}
    Let $R$ be a commutative Bezout domain and a almost zip element of $R$. Then there are only finitely many prime ideals minimal over $aR$.
\end{theorem}

\begin{proof}
   Since $R/\mathrm{rad} (aR)$ is semiprime zip ring then by Theorem~\ref{theor-1} we have that $R/\mathrm{rad}(aR)$ is a Goldie Bezout ring. By Theorem~\ref{theor-6} we have that any minimal prime ideal of $R/\mathrm{rad} (aR)$ is principal and is generated by an idempotent. Then there are only finitely many minimal prime ideals. Obvious then $aR$ has finitely many minimal prime ideals.
\end{proof}

Consequently we have the following  results.

\begin{theorem}\label{theor-8}
    Almost zip commutative Bezout domain is $J$-Noetherian domain (i.e. Noetherian maximal spectrum).
\end{theorem}

\begin{proof}
   By Theorem \ref{theor-8} we have that any nonzero and nonunit element has finitely many minimal ideals. By \cite{8EstesOhm} $R$ is a $J$-Noetherian domain.
\end{proof}

Since a commutative $J$-Noetherian Bezout domain \cite{7ShW} is an elementary divisor ring by Theorem~\ref{theor-8}, we have the following  results.

\begin{theorem}\label{theor-9}
    A commutative almost zip Bezout domain is an elementary divisor domain.
\end{theorem}

Since $J$-Noetherian Bezout domain is fractionally regular ring. We have the following result.

\begin{theorem}\label{theor-10}
    Almost zip Bezout domain is fractionally regular domain.
\end{theorem}

\end{document}